\documentclass[11pt,letterpaper]{amsart}
\usepackage{amsmath,amssymb,amsthm}
\usepackage{enumerate}
\DeclareMathOperator{\Hom}{Hom}
\DeclareMathOperator{\im}{im}
\newcommand{\C}{{\mathbb C}}
\newcommand{\F}{{\mathbb F}}
\newcommand{\R}{{\mathbb R}}
\newcommand{\Z}{{\mathbb Z}}
\newcommand{\card}[1]{\left|{#1}\right|}
\newcommand{\ft}[1]{\hat{#1}}
\newcommand{\conj}[1]{\overline{#1}}
\newcommand{\sums}[1]{\sum_{\substack{#1}}}
\newcommand{\ceil}[1]{\lceil{#1}\rceil}
\newcommand{\floor}[1]{\lfloor{#1}\rfloor}
\newcommand{\Bigfloor}[1]{\Big\lfloor{#1}\Big\rfloor}
\newcommand{\Lr}{{L^r}}
\newcommand{\Linfty}{{L^\infty}}
\newcommand{\Lfour}{{L^4}}
\newcommand{\Ltwo}{{L^2}}
\newcommand{\Ltt}[1]{{||{#1}||_2^2}}
\newcommand{\Ltf}[1]{{||{#1}||_2^4}}
\newcommand{\Lff}[1]{{||{#1}||_4^4}}
\newcommand{\normrat}[1]{{||{#1}||_4/||{#1}||_2}}
\newcommand{\qnormrat}[1]{{\frac{\Lff{#1}}{\Ltf{#1}}}}
\newcommand{\pio}{{p_\iota}}
\newcommand{\qi}{{q_\iota}}
\newcommand{\piset}{{\{\pio:\iota\in I\}}}
\newcommand{\qiset}{{\{\qi:\iota\in I\}}}
\newcommand{\Fq}{{\F_q}}
\newcommand{\Fqi}{{\F_{\qi}}}
\newcommand{\Fqu}{{\F^*_q}}
\newcommand{\Ze}{{\Z^e}}
\newcommand{\Zn}{{\Z^n}}
\newcommand{\ZmZ}{{\Z/m\Z}}
\newcommand{\achars}{{\widehat{\F}_q}}
\newcommand{\mchars}{{\widehat{\Fqu}}}
\newcommand{\Zchars}{{\widehat{\Z}}}
\newcommand{\Zechars}{{\widehat{\Ze}}}
\newcommand{\Znchars}{{\widehat{\Zn}}}
\newcommand{\fio}{f_\iota}
\newcommand{\gio}{g_\iota}
\newcommand{\psii}{\psi_\iota}
\newcommand{\chii}{\chi_\iota}
\newcommand{\alphai}{\alpha_\iota}
\newcommand{\Si}{{S_\iota}}
\newcommand{\ti}{{t_\iota}}
\newcommand{\Vi}{{V_\iota}}
\newcommand{\fields}{\{\Fqi\}_{\iota\in I}}
\newcommand{\pfam}{{\{\fio\}_{\iota \in I}}}
\newcommand{\eseg}{{\{1,\ldots,e\}}}
\newcommand{\sigmas}{{\sigma_1,\ldots,\sigma_e}}
\newcommand{\taus}{{\tau_1,\ldots,\tau_e}}
\newcommand{\grp}{\Gamma}
\newcommand{\grpcard}{\card{\grp}}
\newcommand{\grpchars}{{\widehat{\grp}}}
\newcommand{\subgrp}{K}
\newcommand{\subgrpchars}{{\widehat{\subgrp}}}
\newcommand{\quotind}{{[\grp:\subgrp]}}
\newcommand{\etaprime}{\eta^{\thinspace\prime}}
\newcommand{\kappaprime}{\kappa^\prime}
\newcommand{\lambdaprime}{\lambda^\prime}
\newcommand{\muprime}{\mu^{\thinspace\prime}}
\newcommand{\nuprime}{\nu^{\thinspace\prime}}
\newcommand{\xiprime}{\xi^{\thinspace\prime}}
\newtheorem{theorem}{Theorem}[section]
\newtheorem{proposition}[theorem]{Proposition}
\newtheorem{lemma}[theorem]{Lemma}
\newtheorem{corollary}[theorem]{Corollary}
\theoremstyle{remark}
\newtheorem{step}{Step}
\title{Asymptotic $L^4$ norm of polynomials derived from characters}
\author{Daniel J. Katz}
\thanks{D.J.~Katz is with the Department of Mathematics, Simon Fraser University, 8888 University Drive, Burnaby, BC V5A 1S6, Canada.  Email: {\tt dkatz@sfu.ca}}
\thanks{He is supported by funding from an NSERC grant awarded to Jonathan Jedwab.}
\date{19 October 2012}

\begin{document}
\maketitle
\section*{Abstract}
Littlewood investigated polynomials with coefficients in $\{-1,1\}$ (Littlewood polynomials), to see how small their ratio of norms $||f||_4/||f||_2$ on the unit circle can become as $\deg(f)\to\infty$.
A small limit is equivalent to slow growth in the mean square autocorrelation of the associated binary sequences of coefficients of the polynomials.
The autocorrelation problem for arrays and higher dimensional objects has also been studied; it is the natural generalization to multivariable polynomials.
Here we find, for each $n > 1$, a family of $n$-variable Littlewood polynomials with lower asymptotic $||f||_4/||f||_2$ than any known hitherto.
We discover these through a wide survey, infeasible with previous methods, of polynomials whose coefficients come from finite field characters.
This is the first time that the lowest known asymptotic ratio of norms $||f||_4/||f||_2$ for multivariable polynomials $f(z_1,\ldots,z_n)$ is strictly less than what could be obtained by using products $f_1(z_1) \cdots f_n(z_n)$ of the best known univariate polynomials.
\section{Introduction}

\subsection{History and Main Result}\label{Gerald}

Littlewood pioneered \cite{Littlewood-1966, Littlewood-1968} the study of the $\Lfour$ norm on the complex unit circle of polynomials whose coefficients lie in $\{-1,1\}$, and in particular wanted to know how small their ratio of norms $\normrat{f}$ can become as $\deg(f)\to\infty$.
He suspected, based on calculations of Swinnerton-Dyer, that this ratio could be made to approach $1$ asymptotically, but the smallest limiting ratio he could find was $\sqrt[4]{4/3}$ for the Rudin-Shapiro polynomials \cite{Littlewood-1968}.
The $\Lfour$ norm is of particular interest since it serves as a lower bound for the $\Linfty$ norm and is easier to calculate than most other $\Lr$ norms.
Erd\H os had conjectured that $||f||_\infty/||f||_2$ is bounded away from $1$ for nonconstant polynomials with complex coefficients of unit magnitude \cite[Problem 22]{Erdos-1957}, \cite{Erdos-1962}.
This was disproved by Kahane \cite{Kahane}, but the modified problem where we restrict to coefficients in $\{-1,1\}$ remains open \cite{Newman-Byrnes}, and would be solved if one could prove that $\normrat{f}$ is bounded away from $1$ as $\deg(f)\to\infty$.
Polynomials in one or more variables with coefficients in $\{-1,1\}$ and small $\normrat{f}$ are equivalent to binary sequences and arrays (i.e., those that simply list the coefficients of the polynomials) with low mean square aperiodic autocorrelation.
Such sequences and arrays are important in the theory of communications\footnote{In this milieu, results are expressed in terms of the {\it merit factor}, defined as $\Ltf{f}/(\Lff{f}-\Ltf{f})$.} \cite{Golay-1977} and statistical physics \cite{Bernasconi}.
Accordingly, we define a {\it Littlewood polynomial} in $n$ variables to have the form
\[
f(z_1,\ldots,z_n)=\sum_{j_1=0}^{s_1-1} \cdots \sum_{j_n=0}^{s_n-1} f_{j_1,\ldots,j_n} z_1^{j_1} \cdots z_n^{j_n},
\]
with coefficient $f_{j_1,\ldots,j_n}$ in $\{-1,1\}$, and our $\Lr$ norm for $f(z_1,\ldots,z_n)$ is
\[
||f||_r = \left(\frac{1}{(2\pi)^n} \int_0^{2\pi} \cdots \int_{0}^{2\pi} \big{|}f\big{(}\exp(i\theta_1),\ldots,\exp(i\theta_n)\big{)}\big{|}^r d\theta_1\cdots d\theta_n\right)^{1/r}.
\]
Note that $||f||_2^2=s_1 \cdots s_n$ for our Littlewood polynomial.

For univariate Littlewood polynomials, the lowest asymptotic ratio of norms $\normrat{f}$ found by Littlewood himself \cite{Littlewood-1968} was $\sqrt[4]{4/3}$ for the Rudin-Shapiro polynomials.
Two decades later, this was improved to $\sqrt[4]{7/6}$ by H\o holdt-Jensen \cite{Hoholdt-Jensen}, using modifications of Fekete polynomials.  Over two decades later still, another modification yielded further improvement.
\begin{theorem}[Jedwab-Katz-Schmidt \cite{Jedwab-Katz-Schmidt-a}]\label{Gertrude}
There is a family of univariate Littlewood polynomials which, as $\deg(f)\to\infty$, has $\normrat{f} \to B_1$, the largest real root of $27 x^{12}-498 x^8 +1164 x^4 -722$, which is less than $\sqrt[4]{22/19}$.
\end{theorem}

Prior to this paper, for each $n$, the lowest known asymptotic $\normrat{f}$ for $n$-variable Littlewood polynomials $f(z_1,\ldots,z_n)$ (in the limit as $\deg_{z_1}(f)$, $\ldots$, $\deg_{z_n}(f)\to\infty$) was simply the $n$th power of the lowest known ratio for univariate polynomials, based on the fact that if $f(z_1,\ldots,z_n)=f_1(z_1) \cdots f_n(z_n)$, then $||f||_r=||f_1||_r \cdots ||f_n||_r$.
For bivariate Littlewood polynomials, Schmidt \cite{Schmidt} obtained an asymptotic $\normrat{f}$ of $\sqrt{7/6}$ in this way (via H\o holdt-Jensen's univariate polynomials mentioned above), and foresaw the possibility that the asymptotic $\normrat{f}$ could be lowered to $B_1^2$, contingent upon the conjecture that was later established as Theorem \ref{Gertrude}.
In this paper, we show that one can do better than this product construction, even when based on the best univariate polynomials now known (those of Theorem \ref{Gertrude}).
\begin{theorem}\label{John}
For each $n > 1$, there is a family of $n$-variable Littlewood polynomials $f(z_1,\ldots,z_n)$ which, as $\deg_{z_1}(f),\ldots,\deg_{z_n}(f)\to\infty$, has $\normrat{f}$ tending to a value strictly less than $B_1^n$.
\end{theorem}
The lowest known asymptotic $\normrat{f}$ for $n$-variable Littlewood polynomials is an algebraic number depending on $n$, and is specified precisely in Section \ref{Francis} after we define in Section \ref{Richard} the polynomials that are involved.\footnote{Gulliver and Parker \cite{Gulliver-Parker} have also studied $\normrat{f}$ for multivariable Littlewood polynomials, but in a very different limit: they let the number of variables tend to infinity while keeping the degree in each variable less than or equal to one.}

\subsection{Character Polynomials}\label{Richard}

The polynomials used by H\o holdt-Jensen \cite{Hoholdt-Jensen}, Jedwab-Katz-Schmidt \cite{Jedwab-Katz-Schmidt-a}, and in the current paper to break previous records for the lowest known asymptotic $\normrat{f}$ are all character polynomials, i.e., polynomials whose coefficients are given by characters of finite fields.
The $\Lfour$ norms of character polynomials have already been studied extensively~\cite{Hoholdt-Jensen, Jensen-Hoholdt, Jensen-Jensen-Hoholdt, Boemer-Antweiler, Borwein, Borwein-Choi-2000, Borwein-Choi-2002, Jedwab, Hoholdt, Jedwab-Schmidt, Schmidt}, but it took the new methods of this paper to discover and verify the properties of the polynomials of our Theorem \ref{John}.

The interrelation between the additive and multiplicative structures of finite fields endow character polynomials with their remarkable qualities: the coefficients of an {\it additive character polynomial} are obtained by applying an additive character of a finite field to its nonzero elements arranged multiplicatively (listed as successive powers of a primitive element), while the coefficients of a {\it multiplicative character polynomial} are obtained by applying a multiplicative character of a finite field to its elements arranged additively (as $\Z$-linear combinations of the generators, arrayed in a box whose dimensionality equals the number of generators).
Thus an additive character polynomial has the form
\begin{equation}\label{Samson}
f(z) = \sum_{j \in S} \psi(\alpha(j+t)) z^j,
\end{equation}
where $\psi\colon\Fq \to \C$ is a nontrivial additive character, the {\it support} $S$ is a set of the form $\{0,1,\ldots,s-1\}$, the {\it translation} $t$ is an element of $\Z$, and the {\it arrangement} $\alpha$ is a group epimorphism from $\Z$ to $\Fqu$.
A multiplicative character polynomial has the form
\begin{equation}\label{Raymond}
f(z_1,\ldots,z_e) = \sum_{j=(j_1,\ldots,j_e) \in S} \chi(\alpha(j+t)) z_1^{j_1} \cdots z_e^{j_e},
\end{equation}
where $e$ is a positive integer, $\chi$ is a nontrivial complex-valued multiplicative character of $\Fq=\F_{p^e}$ with $p$ prime, the {\it support} $S$ is $S_1\times\cdots\times S_e$ with each $S_k$ a set of the form $\{0,1,\ldots,s_k-1\}$, while the {\it translation} $t=(t_1,\ldots,t_e)$ is in $\Z^e$, and the {\it arrangement} $\alpha$ is a group epimorphism from $\Ze$ to $\Fq$.
We always extend nontrivial multiplicative characters to take $0$ to $0$.

We now define the Fekete polynomials and their modifications used by H\o holdt-Jensen and Jedwab-Katz-Schmidt.
For an odd prime $p$, the $p$th Fekete polynomial is a multiplicative character polynomial using the quadratic character (Legendre symbol) over the prime field $\F_p$, support $S=\{0,1,\ldots,p-1\}$, translation $t=0$, and arrangement $\alpha\colon \Z\to\F_p$ given by reduction modulo $p$.
Fekete polynomials are themselves the subject of many fascinating studies \cite{Fekete-Polya, Polya, Montgomery, Baker-Montgomery, Conrey-Granville-Poonen-Soundararajan, Borwein-Choi-Yazdani, Borwein-Choi-2002} linking number theory and analysis.
The polynomials used by H\o holdt-Jensen \cite{Hoholdt-Jensen} to obtain asymptotic $\normrat{f}$ of $\sqrt[4]{7/6}$ have the same character, support, and arrangement, but the translations $t$ are chosen such that $t/p \to 1/4$ as $p\to \infty$, and any coefficient of $0$ (arising from the extended multiplicative character) is replaced with $1$ to obtain Littlewood polynomials.
To obtain asymptotic $\normrat{f}$ less than $\sqrt[4]{22/19}$, Jedwab-Katz-Schmidt \cite{Jedwab-Katz-Schmidt-a} use a different limit for $t/p$, and also allow the support $S=\{0,1,\ldots,s-1\}$ to be of size other than $p$, and in fact let $s/p$ tend to a number slightly larger than $1$ as $p\to\infty$.

The families of character polynomials used here are based on similar asymptotics: we say that a family $\pfam$ of additive character polynomials is {\it size-stable} to mean that if we write $\Fqi$ and $\Si$ respectively for the field and support of $\fio$, then $\qiset$ is infinite and $\card{\Si}/(\qi-1)$ tends to a positive real number $\sigma$ (called the {\it limiting size}) as $\qi \to \infty$.
Likewise, we say that a family of $e$-variable multiplicative character polynomials $\pfam$ is {\it size-stable} (resp., {\it translation-stable}) to mean that if we write $\F_{p_\iota^e}$, $\Si=S_{\iota,1} \times \cdots \times S_{\iota,e}$, and $\ti=(t_{\iota,1},\ldots,t_{\iota,e})$ respectively for the field, support, and translation of $\fio$, then the set of primes $\piset$ is infinite and for each $k \in \eseg$, the ratio $\card{S_{\iota,k}}/\pio$ (resp., $t_{\iota,k}/\pio$) tends to a positive real number $\sigma_k$ (resp., a real number $\tau_k$) as $\qi\to\infty$.
We call $\sigmas$ the {\it limiting sizes} and $\taus$ the {\it limiting translations}.

\subsection{Subsidiary Results}\label{Francis}

We discovered the polynomials of Theorem \ref{John} via a survey, enabled by the methods presented in this paper, of the asymptotic $\Lfour$ norms of both additive and multiplicative character polynomials.
Quadratic multiplicative characters behave differently than non-quadratic ones, so we treat them separately: we have {\it quadratic families} in which every character is quadratic, and {\it non-quadratic} families in which none is.
We then have three theorems: one for additive characters and two for the different types of multiplicative characters, and we express our limiting norms in terms of the function 
\begin{equation}\label{Hilda}
\Omega(x,y) = \sum_{n \in \Z} \max(0,1-|x n-y|)^2,
\end{equation}
which is defined and continuous on $\{(x,y) \in \R^2: x\not=0\}$.
\begin{theorem}\label{Gabriel}
Let $\pfam$ be a size-stable family, with limiting size $\sigma$, of additive character polynomials over fields $\fields$.
\begin{enumerate}[(i)]
\item\label{Fabian} As $\qi \to \infty$, we have 
\[
\qnormrat{\fio} \to -\frac{2}{3}\sigma + 2 \, \Omega\left(\frac{1}{\sigma},0\right).
\]
\item\label{James} This limit is globally minimized if and only if $\sigma$ is the unique root in $(1,1+\frac{9}{64})$ of $x^3-12 x+12$.
\end{enumerate}
\end{theorem}
\begin{theorem}\label{Harold}
Let $\pfam$ be a size-stable family, with limiting sizes $\sigma_1$, $\ldots$, $\sigma_e$, of $e$-variable non-quadratic multiplicative character polynomials over fields $\fields$.
\begin{enumerate}[(i)]
\item\label{Yakov} As $\qi\to\infty$, we have 
\[
\qnormrat{\fio} \to -\frac{2^e}{3^e} \prod_{j=1}^e \sigma_i + 2 \prod_{j=1}^e \Omega\left(\frac{1}{\sigma_j},0\right).
\]
\item\label{Karl} This limit is globally minimized if and only if $\sigmas$ all equal the unique root in $(1,1+\frac{3^{e+1}}{2^{2 e+4}})$ of $x^{3 e} -\frac{3^e}{2^{e-3}}(x-1)(3 x^2 -4 x +2)^{e-1}$.
\end{enumerate}
\end{theorem}
\begin{theorem}\label{Earl}
Let $\pfam$ be a size- and translation-stable family, with limiting sizes $\sigmas$ and limiting translations $\taus$, of $e$-variable quadratic multiplicative character polynomials over fields $\fields$.
\begin{enumerate}[(i)]
\item\label{Nancy} As $\qi \to \infty$, we have
\[
\qnormrat{\fio} \to -\frac{2^{e+1}}{3^e} \prod_{j=1}^e \sigma_i + 2 \prod_{j=1}^e \Omega\left(\frac{1}{\sigma_j},0\right) + \prod_{j=1}^e \Omega\left(\frac{1}{\sigma_j},1+\frac{2 \tau_j}{\sigma_j}\right).
\]
\item\label{Ian} This limit is globally minimized if and only if $\sigmas$ all equal the unique root in $(1,1+\frac{3^{e+1}}{2^{2 e+3}})$ of $x^{3 e} -\frac{3^e}{2^{e-2}}(x-1) (3 x^2-4 x+2)^{e-1} - \frac{3^e}{2^{2 e}} (2x-1)^{2 e-1}$, and $\tau_j \in \{ \frac{1-2 \sigma_j}{4} + \frac{n}{2} : n \in \Z\}$ for each $j \in \eseg$.
\end{enumerate}
\end{theorem}
These new theorems are much more general than the compositum of all previous results on the limiting ratio of $\Lfour$ to $\Ltwo$ norm for character polynomials \cite{Hoholdt-Jensen, Jensen-Hoholdt, Jensen-Jensen-Hoholdt, Borwein-Choi-2000, Schmidt, Jedwab-Katz-Schmidt-a, Jedwab-Katz-Schmidt-b}, and reveal for the first time the full functional form of the asymptotic ratio of norms as it depends on choice of character, limiting size, and limiting translation, thus enabling us to find multivariable Littlewood polynomials with lower asymptotic $\normrat{f}$ than any known hitherto.

For each $e \geq 1$, let $A_e$ (resp., $B_e$) be the minimum asymptotic ratio of norms for a family of $e$-variable non-quadratic (resp., quadratic) character polynomials as described in Theorem \ref{Harold}\eqref{Karl} (resp., Theorem \ref{Earl}\eqref{Ian}).
Note that $A_1$ is also the minimum asymptotic ratio of norms achievable by a family of additive character polynomials as described in Theorem \ref{Gabriel}\eqref{James}.
Rational approximations of $B_1^4$, $B_2^4$, $\ldots$, $B_5^4$ are obtained later in Lemma \ref{Henry}, and if desired, a computer may be used to obtain more accurate approximations of values of various $A_e$ and $B_e$.
For each $e\geq 1$, $B_e$ is to date the lowest known asymptotic $\normrat{f}$ for a family of $e$-variable Littlewood\footnote{Quadratic character polynomials are not always Littlewood because the extended quadratic character $\chi$ has $\chi(0)=0$, so we replace each coefficient of $0$ thus produced with a $1$, and Corollary \ref{Oswald} shows that this has no effect on the asymptotic ratio of norms.\label{Eric}} polynomials $f(z_1,\ldots,z_e)$ in the limit as $\deg_{z_1}(f),\ldots,\deg_{z_e}(f) \to \infty$.
For $e=1$, this recapitulates the result of Jedwab-Katz-Schmidt \cite[Corollary 3.2]{Jedwab-Katz-Schmidt-a}, while for $e > 1$, the ratio obtained here is strictly lower than any found to date.
Until now, the smallest known asymptotic ratio has been whatever can be obtained from the best univariate polynomials and the product construction $||f(z_1) \cdots f(z_e)||_r=||f(z)||_r^e$, and so we are claiming that $B_e < B_1^e$ for every $e > 1$.
This will give our main result, Theorem \ref{John}, but in fact we prove something more general: one always obtains a lower asymptotic ratio of norms with a single optimal family of quadratic character polynomials than one does using the product construction with two or more families of character polynomials (which could draw coefficients from $\{-1,1\}$ or a larger set, depending on the characters involved).
\begin{theorem}\label{Orlando}
For each $e\geq 1$, let $A_e$ (resp., $B_e$) be the minimum asymptotic ratio of $\Lfour$ to $\Ltwo$ norm achievable by families of $e$-variable non-quadratic (resp., quadratic) multiplicative character polynomials as described in Theorem \ref{Harold}\eqref{Karl} (resp., Theorem \ref{Earl}\eqref{Ian}).
Then $B_e < A_e$ for every $e \geq 1$ and $B_{e_1+e_2} < B_{e_1} B_{e_2}$ for every $e_1,e_2 \geq 1$.
\end{theorem}

\subsection{Organization of this Paper}

To prove Theorems \ref{Gabriel}--\ref{Earl}, we first establish a general theorem for obtaining the $\Lfour$ norm of a polynomial from its Fourier interpolation in Section \ref{Lawrence}, after setting down notational conventions in Section \ref{William}.
Our general theorem reduces the problem of computing $\Lfour$ norms of character polynomials to a pair of calculations (one for additive and one for multiplicative characters) involving Gauss sums, which are presented in Section \ref{Odysseus}.
We use these in Section \ref{Tamar} to prove Theorems \ref{Gabriel}\eqref{Fabian}, \ref{Harold}\eqref{Yakov}, and \ref{Earl}\eqref{Nancy}.
These respectively imply Theorems \ref{Gabriel}\eqref{James}, \ref{Harold}\eqref{Karl}, and \ref{Earl}\eqref{Ian}, but showing this demands delicate arguments which are sketched in Section \ref{Valerie}.
We prove Theorem \ref{Orlando} in Section \ref{Albert}.
Some technical results used in Sections \ref{Tamar} and \ref{Valerie} are collected and proved in the Appendix.

\section{Notation and Conventions}\label{William}

For the rest of this paper $p$ is a prime, and $q=p^e$ with $e$ a positive integer.
For any group $\grp$, we use $\grpchars$ to denote the group of characters from $\grp$ into $\C$: thus $\achars$ is the group of additive characters from $\Fq$ to $\C$ and $\mchars$ the group of multiplicative characters from $\Fqu$ to $\C$.
We extend any nontrivial $\chi \in \mchars$ so that $\chi(0)=0$.

To write the multiplicative character polynomial \eqref{Raymond} compactly, we use the convention that if $j=(j_1,\ldots,j_e) \in \Ze$, the notation $z^j$ is a shorthand for $z_1^{j_1} \cdots z_e^{j_e}$.
To make it easier to speak about supports of character polynomials \eqref{Samson} and \eqref{Raymond}, we call a finite set of consecutive integers a {\it segment}, and a finite Cartesian product of segments a {\it box}.
If $S$ is a subset of $\Z^n$ and $t \in \Z^n$, then $S+t$ is the translated subset $\{s+t: s \in S\}$.

\section{$\Lfour$ Norms via the Fourier Transform}\label{Lawrence}

If $\grp$ is a finite abelian group and $\{F_g\}_{g \in \grp}$ is a family of complex numbers, then for any $\eta\in\grpchars$, we have the Fourier transform
\[
\ft{F}_\eta = \sum_{g \in \grp} F_g \eta(g),
\]
with inverse
\[
F_g = \frac{1}{\grpcard} \sum_{\eta \in \grpchars} \ft{F}_\eta \conj{\eta(g)}.
\]
We express the $\Lfour$ norm in terms of the Fourier interpolation.
\begin{theorem}
Let $\grp$ be a finite abelian group, $\{F_g\}_{g \in \grp}$ a family of complex numbers, $n$ a positive integer, and $\pi\in\Hom(\Z^n,\grp)$.
For any $\eta\in\grpchars$, let $\etaprime \in \Znchars$ be $\eta \circ \pi$.
If $U$ is a finite subset of $\Zn$ and $F(z)=\sum_{u \in U} F_{\pi(u)} z^u \in \C[z_1,\ldots,z_n]$, then
\[
\Lff{F}  = \frac{1}{\grpcard^5} \sums{a,b,c,d \in U \\ a+b=c+d} \,\, \sum_{\kappa,\lambda,\mu,\nu \in \grpchars} \conj{\kappaprime(a) \lambdaprime(b)} \muprime(c) \nuprime(d) H(\kappa,\lambda,\mu,\nu),
\]
where
\[
H(\kappa,\lambda,\mu,\nu) = \sum_{\xi \in \grpchars} \ft{F}_{\xi\kappa} \ft{F}_{\xi\lambda} \conj{\ft{F}_{\xi\mu} \ft{F}_{\xi\nu}}.
\]
\end{theorem}
\begin{proof}
By the definition of the $\Lfour$ norm, we have
\[
\Lff{F}  = \sums{a,b,c,d \in U \\ a+b=c+d} F_{\pi(a)} F_{\pi(b)} \conj{F_{\pi(c)} F_{\pi(d)}},
\]
and thus, using the inverse Fourier transform
\[
\Lff{F}  = \frac{1}{\grpcard^4} \sums{a,b,c,d \in U \\ a+b=c+d} \,\, \sum_{\kappa,\lambda,\mu,\nu \in \grpchars} \ft{F}_\kappa \ft{F}_\lambda \conj{\ft{F}_\mu \ft{F}_\nu \kappaprime(a) \lambdaprime(b)} \muprime(c) \nuprime(d).
\]
Since we are summing $\kappa$ over all $\grpchars$, we can replace $\kappa$ by $\xi\kappa$ for any given $\xi\in\grpchars$, and also do likewise with $\lambda$, $\mu$, $\nu$ to obtain
\[
\Lff{F}  = \frac{1}{\grpcard^4} \sums{a,b,c,d \in U \\ a+b=c+d} \,\, \sum_{\kappa,\lambda,\mu,\nu \in \grpchars} \ft{F}_{\xi\kappa} \ft{F}_{\xi\lambda} \conj{\ft{F}_{\xi\mu} \ft{F}_{\xi\nu} \kappaprime(a) \lambdaprime(b)} \muprime(c) \nuprime(d)
\]
where we have omitted mention of the resulting factor of $\conj{\xiprime(a)\xiprime(b)}\xiprime(c)\xiprime(d)$, which equals $1$ in view of the constraint in the first summation.
Now sum $\xi$ over $\grpchars$ and divide by $\grpcard=|\grpchars|$ to finish.
\end{proof}
We apply this general theorem to additive and multiplicative character polynomials in two corollaries below.
Such polynomials have Gauss sums as their Fourier coefficients, so for any $\psi\in\achars$ and $\chi\in\mchars$, we define the {\it Gauss sum} associated with $\psi$ and $\chi$ to be
\[
G(\psi,\chi)=\sum_{a \in \Fqu} \psi(a) \chi(a).
\]
\begin{corollary}\label{Martin}
If $f(z)$ is an additive character polynomial with character $\psi\in\achars$, support $S$, translation $t$, and arrangement $\alpha$, then
\[
\Lff{f} = \frac{1}{(q-1)^5} \sums{a,b,c,d \in S+t \\ a+b=c+d} \,\, \sum_{\kappa,\lambda,\mu,\nu \in \mchars} \conj{\kappaprime(a) \lambdaprime(b)} \muprime(c) \nuprime(d) H(\kappa,\lambda,\mu,\nu)
\]
where for any $\eta\in\mchars$, we let $\etaprime\in\Zchars$ be $\eta\circ\alpha$, and
\[
H(\kappa,\lambda,\mu,\nu) = \sum_{\xi \in \mchars} G(\psi,\xi\kappa) G(\psi,\xi\lambda) \conj{G(\psi,\xi\mu) G(\psi,\xi\nu)}.
\]
\end{corollary}
\begin{proof}
Our additive character polynomial $f(z)=\sum_{s \in S} \psi(\alpha(s+t)) z^s$ has the same $\Lr$ norms as $F(z)=z^t f(z)=\sum_{u \in S+t} \psi(\alpha(u)) z^u$, so take $\grp=\Fqu$, $F_g=\psi(g)$, $n=1$, $\pi=\alpha$, and $U=S+t$ in the above theorem, and note that for $\eta\in\mchars$ we have $\ft{F}_\eta =G(\psi,\eta)$.
\end{proof}
\begin{corollary}\label{Nathan}
If $f(z)$ is a multiplicative character polynomial with character $\chi\in\mchars$, support $S$, translation $t$, and arrangement $\alpha$, then
\[
\Lff{f} = \frac{1}{q^5} \sums{a,b,c,d \in S+t \\ a+b=c+d} \,\, \sum_{\kappa,\lambda,\mu,\nu \in \achars} \conj{\kappaprime(a) \lambdaprime(b)} \muprime(c) \nuprime(d) H(\kappa,\lambda,\mu,\nu)
\]
where for any $\eta\in\achars$, we let $\etaprime\in\Zechars$ be $\eta\circ\alpha$, and
\[
H(\kappa,\lambda,\mu,\nu) = \sum_{\xi \in \achars} G(\xi\kappa,\chi) G(\xi\lambda,\chi) \conj{G(\xi\mu,\chi) G(\xi\nu,\chi)}.
\]
\end{corollary}
\begin{proof}
Our multiplicative character polynomial $f(z)=\sum_{s \in S} \chi(\alpha(s+t)) z^s$ has the same $\Lr$ norms as $F(z)=z^t f(z)=\sum_{u \in S+t} \chi(\alpha(u)) z^u$, so take $\grp=\Fq$, $F_g=\chi(g)$, $n=e$, $\pi=\alpha$, and $U=S+t$ in the above theorem, and note that for $\eta\in\achars$ we have $\ft{F}_\eta =G(\eta,\chi)$ .
\end{proof}
The key to $\Lfour$ norms is then the evaluation of the sums $H(\kappa,\lambda,\mu,\nu)$ in the above two corollaries, which we take up in the next section.

\section{Two Propositions on Summations of Gauss Sums}\label{Odysseus}

Here we estimate the values of the summations $H$ that appear in Corollaries \ref{Martin} and \ref{Nathan}.
We begin with some basic facts about Gauss sums, which are proved in Theorems 5.11 and 5.12 of \cite{Lidl-Niederreiter}.
\begin{lemma}\label{Paul} If $\psi\in\achars$ and $\chi\in\mchars$, then
\begin{enumerate}[(i)]
\item $G(\psi,\chi)=q-1$ if both characters are trivial,
\item\label{Reginald} $G(\psi,\chi)=0$ if $\psi$ is trivial and $\chi$ is not,
\item\label{Raphael} $G(\psi,\chi)=-1$ if $\chi$ is trivial and $\psi$ is not,
\item\label{Samuel} $|G(\psi,\chi)|=\sqrt{q}$ if both characters are nontrivial, and
\item\label{Tantalus} $\sum_{a \in \Fqu} \psi(b a) \chi(a)=\conj{\chi}(b) G(\psi,\chi)$ for any $b \in \Fqu$.
\end{enumerate}
\end{lemma}
We first estimate the summation $H$ appearing in Corollary \ref{Martin}.
\begin{proposition}\label{Vaclav}
Let $\psi$ be a nontrivial character in $\achars$, and $\kappa,\lambda,\mu,\nu \in \mchars$.  If
\[
H = \sum_{\xi\in\mchars} G(\psi,\xi\kappa) G(\psi,\xi\lambda) \conj{G(\psi,\xi\mu) G(\psi,\xi\nu)},
\]
and
\[
M
=
\begin{cases}
(q-1)^3 & \text{if $\{\kappa,\lambda\}=\{\mu,\nu\}$,} \\
0       & \text{otherwise,}
\end{cases}
\]
then $|H-M| \leq (q-1)q\sqrt{q}$.
\end{proposition}
\begin{proof}
First we consider the case where $\{\kappa,\lambda\}=\{\mu,\nu\}$, wherein
\[
H = \sum_{\xi\in\mchars} |G(\psi,\xi\kappa)|^2 |G(\psi,\xi\lambda)|^2.
\]
One can work out from Lemma \ref{Paul}\eqref{Raphael},\eqref{Samuel} that $H=(q-2)q^2+1$ if $\kappa=\lambda=\mu=\nu$ and $H=(q-3)q^2+2 q$ otherwise.
Thus $H-M=(q-1)(q-2)$ or $1-q$.

Now we consider the case where $\{\kappa,\lambda\}\not=\{\mu,\nu\}$, wherein
\begin{align*}
H 
& = \sum_{\xi\in\mchars} \sum_{w,x,y,z \in \Fqu} \psi(w+x-y-z) \xi(w x y^{-1} z^{-1}) \kappa(w) \lambda(x) \conj{\mu(y) \nu(z)} \\
& = (q-1) \sums{w,x,y,z \in \Fqu \\ w x = y z} \psi(w+x-y-z) \kappa(w) \lambda(x) \conj{\mu(y) \nu(z)}.
\end{align*}
Now reparameterize the sum with $w=u y$ and $z=u x$ to obtain
\[
H = (q-1) \sum_{u,x,y \in \Fqu} \conj{\psi}((u-1)x) \psi((u-1)y) \kappa\conj{\nu}(u) \lambda\conj{\nu}(x) \kappa\conj{\mu}(y),
\]
and since $\{\kappa, \lambda\}\not=\{\mu, \nu\}$, we can restrict to $u\not=1$ without changing the value of the summation.
Then Lemma \ref{Paul}\eqref{Tantalus} tells us that when we sum over $x$ and $y$, we obtain
\[
H = (q-1) G(\psi,\kappa\conj{\mu}) G(\conj{\psi},\lambda\conj{\nu}) \sum_{u\not=0,1} \conj{\kappa\lambda}\mu\nu(u-1) \kappa\conj{\nu}(u).
\]
Now $\conj{\kappa\lambda}\mu\nu$ and $\kappa\conj{\nu}$ can not both be the trivial character since $\{\kappa,\lambda\}\not=\{\mu,\nu\}$.
If $\conj{\kappa\lambda}\mu\nu$ is trivial, then the sum over $u$ is $-1$; if $\kappa\conj{\nu}$ is trivial, the sum is $-\conj{\kappa\lambda}\mu\nu(-1)$; otherwise, let $\omega$ be a generator of $\mchars$ and we can write the sum over $u$ as $\sum_{u\not=0,1} \omega((u-1)^a u^b)$ for some nonzero $a,b \in \Z/(q-1)\Z$, and use the Weil bound \cite{Weil}, \cite[Theorem 5.41]{Lidl-Niederreiter} to see that this sum is bounded in magnitude by $\sqrt{q}$.
We can use this fact, along with Lemma \ref{Paul}\eqref{Raphael}, \eqref{Samuel}, to see that $|H|\leq (q-1)q\sqrt{q}$.
\end{proof}
Similarly, we estimate the summation $H$ appearing in Corollary \ref{Nathan}.
\begin{proposition}\label{Walter}
Let $\chi$ be a nontrivial character in $\mchars$, and $\kappa,\lambda,\mu,\nu \in \achars$.  If
\[
H = \sum_{\xi\in\achars} G(\xi\kappa,\chi) G(\xi\lambda,\chi) \conj{G(\xi\mu,\chi) G(\xi\nu,\chi)},
\]
and
\[
M
=
\begin{cases}
q^3 & \text{if $\{\kappa,\lambda\}=\{\mu,\nu\}$,} \\
q^3 & \text{if $\kappa=\lambda$, $\mu=\nu$, and $\chi$ is the quadratic character,} \\
0   & \text{otherwise,}
\end{cases}
\]
then $|H-M| \leq 3 q^2 \sqrt{q}$.
\end{proposition}
\begin{proof}
Let $\epsilon$ be the canonical additive character over $\Fq$.
Then for any $\eta\in\achars$, there is a unique $y \in \Fq$ such that $\eta(z)=\epsilon(y z)$ for all $z \in \Fq$.
Let $a,b,c,d$ be chosen so that $\kappa(z)=\epsilon(a z)$, $\lambda(z)=\epsilon(b z)$, $\mu(z)=\epsilon(c z)$, and $\nu(z)=\epsilon(d z)$ for all $z \in \Fq$.
Furthermore, we shall parameterize the sum of $\xi$ over $\achars$ in the definition of $H$ by a sum over $x \in \Fq$, and replace $\xi(z)$ with $\epsilon(x z)$ wherever it occurs.
Thus, in view Lemma \ref{Paul}\eqref{Tantalus} and \eqref{Reginald}, we have
\[
H = |G(\epsilon,\chi)|^4 \sum_{x\in\Fq} \conj{\chi}((x+a)(x+b))\chi((x+c)(x+d)),
\]
and $|G(\epsilon,\chi)|=\sqrt{q}$ by Lemma \ref{Paul}\eqref{Samuel}.
Let $m$ be the order of $\chi$.  Then
\[
H = q^2 \sum_{x \in \Fq} \chi((x+a)^{m-1} (x+b)^{m-1} (x+c) (x+d)).
\]
The magnitude of the Weil sum over $x$ is bounded by $3\sqrt{q}$ unless the polynomial $(x+a)^{m-1} (x+b)^{m-1} (x+c) (x+d)$ is an $m$th power in $\Fq[x]$.  (See \cite{Weil}, \cite[Theorem 5.41]{Lidl-Niederreiter}.)
It is an $m$th power only if $\{a,b\}=\{c,d\}$ or if $m=2$, $a=b$, and $c=d$, in which cases the Weil sum is either $q-1$ (if $a=b=c=d$) or $q-2$ (if there are two distinct roots).
\end{proof}

\section{Asymptotic $\Lfour$ Norm}\label{Tamar}
We prove Theorems \ref{Gabriel}\eqref{Fabian}, \ref{Harold}\eqref{Yakov}, and \ref{Earl}\eqref{Nancy} in Section \ref{Tamar} by using the propositions from the previous section with Corollaries \ref{Martin} and \ref{Nathan}.
\subsection{Proof of Theorems \ref{Harold}(\ref{Yakov}) and \ref{Earl}(\ref{Nancy})}
Let $\chi$ be a nontrivial character in $\mchars$, let $\alpha$ be an epimorphism from $\Ze$ to $\Fq$, let $t \in \Z^e$, and let $S=S_1\times\cdots\times S_e$ be a box where each $S_j$ is a nonempty segment of the form $\{0,1,\ldots,s_j-1\}$.
Recall from Section \ref{William} our notational convention that $z^s$ is shorthand for $z_1^{s_1} \ldots z_e^{s_e}$ when $s=(s_1,\ldots,s_e) \in \Z^e$.
Let $f(z)$ be the multiplicative character polynomial $\sum_{s \in S} \chi(\alpha(s+t)) z^s$.
We shall calculate $\Lff{f}$ first, and then investigate what happens asymptotically to this quantity in the limits considered in Theorems \ref{Harold}\eqref{Yakov} and \ref{Earl}\eqref{Nancy}.
By Corollary \ref{Nathan}, we have
\begin{equation}\label{Zachary}
\Lff{f} = \frac{1}{q^5} \sums{a,b,c,d \in U \\ a+b=c+d} \,\, \sum_{\kappa,\lambda,\mu,\nu \in \achars} \conj{\kappaprime(a) \lambdaprime(b)} \muprime(c) \nuprime(d) H(\kappa,\lambda,\mu,\nu),
\end{equation}
where we let $U=S+t$, and for any $\eta\in\achars$, we let $\etaprime=\eta\circ\alpha$, and
\[
H(\kappa,\lambda,\mu,\nu) = \sum_{\xi\in \achars} G(\xi\kappa,\chi) G(\xi\lambda,\chi) \conj{G(\xi\mu,\chi) G(\xi\nu,\chi)}.
\]
By Proposition \ref{Walter}, we can write $H(\kappa,\lambda,\mu,\nu)=M(\kappa,\lambda,\mu,\nu)+N(\kappa,\lambda,\mu,\nu)$ with
\[
M(\kappa,\lambda,\mu,\nu)
=
\begin{cases}
q^3 & \text{if $\{\kappa,\lambda\}=\{\mu,\nu\}$,} \\
q^3 & \text{if $\kappa=\lambda$, $\mu=\nu$, and $\chi$ is the quadratic character,} \\
0   & \text{otherwise,}
\end{cases}
\]
and
\begin{equation}\label{Pamela}
|N(\kappa,\lambda,\mu,\nu)| \leq 3 q^2 \sqrt{q}
\end{equation}
for all $\kappa$, $\lambda$, $\mu$, $\nu\in\mchars$.

If $\chi$ is non-quadratic, when we write out separately the contributions from $M$ and $N$ to \eqref{Zachary}, we get $\Lff{f} = A+B-D+E$, where
\begin{align*}
A & = \frac{1}{q^2} \sums{a,b,c,d \in U \\ a+b=c+d} \,\, \sum_{\kappa,\lambda \in \achars} \kappaprime(c-a) \lambdaprime(d-b), \\
B & = \frac{1}{q^2} \sums{a,b,c,d \in U \\ a+b=c+d} \,\, \sum_{\kappa,\lambda \in \achars} \kappaprime(d-a) \lambdaprime(c-b), \\
D & = \frac{1}{q^2} \sums{a,b,c,d \in U \\ a+b=c+d} \,\, \sum_{\kappa \in \achars} 1, \\
E & = \frac{1}{q^5} \sum_{\kappa,\lambda,\mu,\nu \in \achars} N(\kappa,\lambda,\mu,\nu) \sums{a,b,c,d \in U \\ a+b=c+d} \conj{\kappaprime(a) \lambdaprime(b)} \muprime(c) \nuprime(d).
\end{align*}
Here $A$ accounts for the value of $M$ when $(\kappa,\lambda)=(\mu,\nu)$, and $B$ accounts for the value of $M$ when $(\kappa,\lambda)=(\nu,\mu)$, while $D$ corrects for the double counting by $A$ and $B$ of the case $\kappa=\lambda=\mu=\nu$.

Note that $A=B$, and that $A$ counts the number of $(a,b,c,d) \in U^4$ with $c-a=b-d \in \ker\alpha$.
If we write $a=(a_1,\ldots,a_e)$, $b=(b_1,\ldots,b_e)$, $c=(c_1,\ldots,c_e)$, and $d=(d_1,\ldots,d_e)$, then $c-a \in \ker\alpha$ is equivalent to $c_1-a_1 \equiv \cdots \equiv c_e-a_e \equiv 0 \pmod{p}$, because $\alpha$ is an epimorphism from $\Ze$ to $\Fq=\F_{p^e}$ and so factors as $\alpha=\gamma \circ \beta$, with $\beta\colon \Ze\to(\Z/p\Z)^e$ coordinate-wise reduction modulo $p$ and $\gamma\colon (\Z/p\Z)^e \to \Fq$ a group isomorphism.
Now $U=U_1\times\cdots\times U_e$ with each $U_j=\{t_j,t_j+1,\ldots,t_j+|S_j|-1\}$, so for each $n \in \Z$, there are $\max(0,\card{S_j}-p|n|)$ ways for $c_j-a_j$ to equal $p n$ and the same number of ways for $b_j-d_j$ to equal $p n$.
So
\[
A = B = \prod_{j=1}^e \sum_{n_j \in \Z} \max(0,\card{S_j}-p |n_j|)^2.
\]

On the other hand, $q D$ counts the number of $(a,b,c,d) \in U^4$ with $c-a=b-d$, so by the same argument we just used (with modulus $1$ instead of $p$), $q D = \prod_{j=1}^e \sum_{n_j \in \Z} \max(0,\card{S_j}-|n_j|)^2$, from which we can compute
\[
D = \prod_{j=1}^e \left(\frac{2\card{S_j}^3+\card{S_j}}{3 p}\right).
\]
Now we bound $E$ via two bounds: (i) our bound \eqref{Pamela} on $N$, and (ii) a technical result, Lemma \ref{Xavier} in the Appendix, bounding the inner sum of $E$.
We satisfy the condition on $\alpha$ demanded by this lemma, since $\alpha=\gamma \circ \beta$ with $\beta\colon \Ze\to(\Z/p\Z)^e$ coordinate-wise reduction modulo $p$ and $\gamma \colon (\Z/p\Z)^e \to \Fq$ a group isomorphism.  With these two bounds, we obtain
\begin{equation}\label{Eddie}
|E| \leq 3 \cdot 64^e q \sqrt{q} \prod_{j=1}^e \max\left(1,\frac{\card{S_j}}{p}\right)^3 \prod_{j=1}^e (1+\log p)^3.
\end{equation}
Now we divide $\Lff{f}=A+B-D+E$ by $\Ltf{f}$ and consider the limit where each $\card{S_j}/p \to \sigma_j$ as $q\to\infty$, that is, consider what happens in a size-stable family of polynomials.
Another technical result, Lemma \ref{Caroline} in the Appendix, shows that we can replace the denominator $\Ltf{f}$ with $\card{S}^2$ without changing the limit.  Then recall the definition \eqref{Hilda} of $\Omega$, and note that $A/\card{S}^2$ and $B/\card{S}^2$ tend to $\prod_{j=1}^e \Omega(1/\sigma_j,0)$, that $D/\card{S}^2$ tends to $(2/3)^e \prod_{j=1}^e \sigma_j$, and that $|E|/\card{S}^2$ tends to $0$ in this limit.

If $\chi$ is quadratic, the proof is done in the same manner, except that there is now a contribution from $M$ in the case where $\kappa=\lambda$ and $\mu=\nu$, and so we get $\Lff{f} = A+B+C-2 D+E$, where $A$, $B$, $D$, and $E$ are as defined above, and
\[
C = \frac{1}{q^2} \sums{a,b,c,d \in U \\ a+b=c+d} \,\, \sum_{\kappa,\mu \in \achars} \kappaprime(-a-b) \muprime(c+d).
\]
Note that we subtract $D$ twice now because $A$, $B$, and $C$ count the case where $\kappa=\lambda=\mu=\nu$ three times.
$C$ counts the number of $(a,b,c,d)\in U^4$ with $a+b=c+d\in\ker\alpha$.
Following the method we used to determine $A$, write $a=(a_1,\ldots,a_e)$, $b=(b_1,\ldots,b_e)$, $c=(c_1,\ldots,c_e)$, and $d=(d_1,\ldots,d_e)$, and note that $a+b \in \ker\alpha$ is equivalent to $a_1+b_1 \equiv \cdots \equiv a_e+b_e \equiv 0 \pmod{p}$.
Since $U=U_1\times\cdots\times U_e$ with each $U_j=\{t_j,t_j+1,\ldots,t_j+|S_j|-1\}$, there are $\max(0,\card{S_j}-|n p - (2 t_j + \card{S_j} - 1)|)$ ways to obtain $a_j+b_j=n p$ with $(a_j,b_j) \in U_j^2$, and the same number of ways to obtain $c_j+d_j=n p$ with $(c_j,d_j) \in U_j^2$, so 
\[
C=\prod_{j=1}^e \sum_{n_j \in \Z} \max(0,\card{S_j}-|p n_j -\card{S_j}-2 t_j+1|)^2,
\]
and if we have both size- and translation-stability, then $\card{S_j}/p \to\sigma_j$ and $t_j/p \to \tau_j$ as $q \to \infty$, so that $C/\card{S}^2 \to \prod_{j=1}^e \Omega\left(1/\sigma_j,1+2\tau_j/\sigma_j\right)$.
\subsection{Proof of Theorem \ref{Gabriel}(\ref{Fabian})}
The proof is the same, {\it mutatis mutandis}, as for the $e=1$ case of Theorem \ref{Harold}\eqref{Yakov}, with the roles of $\Fq$ and $\Fqu$ exchanged.
Corollary \ref{Martin} and Proposition \ref{Vaclav} replace Corollary \ref{Nathan} and Proposition \ref{Walter}, and Lemma \ref{Caroline} becomes unnecessary as $\Ltt{f}$ for an additive character polynomial $f$ is always precisely equal to the cardinality of the support of $f$.
These, and other attendant minor changes resulting from the exchange of $\Fq$ and $\Fqu$, cause \eqref{Pamela} to become $|N(\kappa,\lambda,\mu,\nu)| \leq (q-1) q\sqrt{q}$, and \eqref{Eddie} to become $|E| \leq 64 q \sqrt{q} \max\left(1,\card{S}/(q-1)\right)^3 (1+\log(q-1))^3$, and any other printed instance of $p$ or $q$ should be replaced with $q-1$.

\section{Minimizing the Asymptotic Ratio of $\Lfour$ to $\Ltwo$ Norm}\label{Valerie}

Here we prove Theorems \ref{Gabriel}\eqref{James}, \ref{Harold}\eqref{Karl}, and \ref{Earl}\eqref{Ian} by finding the limiting sizes and (for quadratic multiplicative character polynomials) the limiting translations that globally minimize the ratio of the $\Lfour$ to $\Ltwo$ norm.

\subsection{Proof of Theorem \ref{Harold}(\ref{Karl})}
In view of Theorem \ref{Harold}\eqref{Yakov}, we are trying to minimize the limiting ratio of norms, given by the function
\[
K(x_1,\ldots,x_e)=-\frac{2^e}{3^e} \prod_{j=1}^e x_i + 2 \prod_{j=1}^e \Phi(x_j),
\]
for $x_1,\ldots,x_e$ positive real numbers (the limiting sizes), where for positive $x$, we define
\begin{equation}\label{Mary}
\Phi(x)=\Omega\left(\frac{1}{x},0\right)=\sum_{n \in \Z} \max\left(0,1-\frac{|n|}{x}\right)^2,
\end{equation}
which is differentiable for $x\not=0$ and is $C^\infty$ for $x\not\in\Z$.
\begin{step}
We can assume that each $x_j > 1$ because otherwise the partial derivative of $K$ with respect to $x_j$ would be negative.
\end{step}
\begin{step}
We can assume that $(x_1,\ldots,x_e) \in (1,3)^e$: Lemma \ref{Laura} in the Appendix shows that $K(x_1,\ldots,x_e) \geq \Phi(x_1)\cdots\Phi(x_e)$, and note that $\Phi(x)$ is increasing for $x > 1$, that $\Phi(1)=1$, $\Phi(3)=19/9 > 2$, and $K(1,\ldots,1)=2-(2/3)^e < 2$.
This proves that a global minimum exists and lies in $(1,3)^e$: the closure of $(1,3)^e$ is compact and $K$ is continuous thereupon.
\end{step}
\begin{step}
Suppose $(\sigmas)$ to be global minimizer of $K$.
Then the partial derivatives of $K$ must vanish there, whence for each $k \in \eseg$, we have $2 u(\sigma_k) \prod_{j=1}^e U(\sigma_j) = 1$, where $u(x)=\frac{x \Phi^\prime(x)}{\Phi(x)}$ and $U(x)=\frac{3 \Phi(x)}{2 x}$ for $x > 0$.
\end{step}
\begin{step}
Then one can show that $U(x)$ is strictly decreasing on $[1,3]$, with $U(3)=19/18$.
Thus we must have $u(\sigma_k) \leq (1/2)\cdot(18/19)^e < 1/2$ for all $k \in \eseg$.
Then examination of $u(x)$ shows that $u(x)$ strictly increases from $0$ to $1/2$ for $x \in [1,2-\sqrt{2/3}]$, and then $u(x) > 1/2$ for $x \in (2-\sqrt{2/3},3)$.
This then forces $\sigma_1=\cdots=\sigma_e < 2 -\sqrt{2/3} < 6/5$.
\end{step}
\begin{step}
Now $U(\sigma_1) > U(6/5) > 9/7$, so this forces $u(\sigma_1) < (1/2)\cdot(7/9)^e \leq 7/18$, which in turn forces $\sigma_1 < 8/7$.
Then $U(\sigma_1) > U(8/7) > 4/3$, so this forces $u(\sigma_1) < (1/2)\cdot (3/4)^e$.
Since $u(x) \geq 8(x-1)/3$ for $x \in [1,8/7]$, this forces $\sigma_1 < 1+\frac{3^{e+1}}{2^{2 e+4}}$.
\end{step}
\begin{step}
Now our problem is reduced to the single-variable minimization of $\Theta(x)=K(x,\ldots,x)=-\left(\frac{2 x}{3}\right)^e + 2 \Phi(x)^e$ on the interval $(1,1+\frac{3^{e+1}}{2^{2 e+4}})$.
It is not hard to see that $d\Theta/d x$ vanishes if and only if $x^{3 e} -\frac{3^e}{2^{e-3}}(x-1)(3 x^2 -4 x +2)^{e-1}$ vanishes.
Meanwhile $d^2\Theta/d x^2 > 0$ on our interval: by computing its value and then dropping a nonnegative term, we can see that $d^2\Theta/d x^2$ is at least $-e(e-1)\frac{2^e x^{e-2}}{3^e} + 8 e \frac{3-2 x}{x^4} \Phi(x)^{e-1} \geq -e(e-1)\frac{19^e}{24^e} + 2 e > 0$.
This proves that there is a unique minimum: the unique root $a_e$ of $x^{3 e} -\frac{3^e}{2^{e-3}}(x-1)(3 x^2 -4 x +2)^{e-1}$ lying in $(1,1+\frac{3^{e+1}}{2^{2 e+4}})$.
\end{step}
\subsection{Proof of Theorem \ref{Gabriel}(\ref{James})}
This is accomplished exactly as the $e=1$ case of the proof of Theorem \ref{Harold}\eqref{Karl} above, save that Lemma \ref{Katherine} replaces Lemma \ref{Laura}.
\subsection{Proof of Theorem \ref{Earl}(\ref{Ian})}\label{George}
In view of Theorem \ref{Earl}\eqref{Nancy}, we are trying to minimize the limiting ratio of norms, given by the function
\begin{equation}\label{Abigail}
-\frac{2^{e+1}}{3^e} \prod_{j=1}^e x_j + 2 \prod_{j=1}^e \Omega\left(\frac{1}{x_j},0\right) + \prod_{j=1}^e \Omega\left(\frac{1}{x_j},1+\frac{2 y_j}{x_j}\right)
\end{equation}
for $x_1,\ldots,x_e$ positive real numbers (the limiting sizes) and $y_1,\ldots,y_e$ arbitrary real numbers (the limiting translations).
\setcounter{step}{0}
\begin{step}
We invoke Lemma \ref{Faye}\eqref{Daphne} in the Appendix to see that we can confine our search to $x_1,\ldots,x_e \geq 1/2$.
For as long as $x_j \leq 1/2$, the lemma shows that we can always arrange for $y_j$ to be such that $\Omega(x_j^{-1},1+2 x_j^{-1} y_j)=0$, and we note that $\Omega(x_j^{-1},0)=1$ for all $x_j \in (0,1/2]$.
Thus we can increase $x_j$ to $1/2$ to lower \eqref{Abigail} through the $-\frac{2^{e+1}}{3^e} \prod_{j=1}^e x_j$ term while keeping the other terms constant.
\end{step}
\begin{step}
Now we invoke Lemma \ref{Faye}\eqref{Edith} from the Appendix to see that for fixed $x_1,\ldots,x_e$, we minimize the last term of \eqref{Abigail} if and only if we arrange that $y_j \in \{\frac{1 - 2 x_j}{4} + \frac{m}{2} : m \in \Z\}$ for each $j \in \eseg$.
The problem is thus reduced to the minimization of
\[
\Lambda(x_1,\ldots,x_e)=-\frac{2^{e+1}}{3^e} \prod_{j=1}^e x_j + 2 \prod_{j=1}^e \Phi(x_j) + \prod_{j=1}^e \Psi(x_j)
\]
for $x_1,\ldots,x_e$ positive real numbers, where $\Phi(x)$ is as defined in \eqref{Mary}, and 
\[
\Psi(x)=\sum_{n \in \Z} \max\left(0,1-\frac{|2 n+1|}{2 x}\right)^2
\]
for $x>0$.
Note that $\Psi$ is differentiable for $x\not=0$ and is $C^\infty$ for $x\not\in\Z+1/2$.
\end{step}
\begin{step}
We can assume that each $x_j > 1/2$ because otherwise the partial derivative of $\Lambda$ with respect to $x_j$ would be negative.
\end{step}
\begin{step}
We can assume that $x_1,\ldots,x_e \in (1/2,3)$: Lemma \ref{Laura} in the Appendix shows that $\Lambda(x_1,\ldots,x_e) \geq \Phi(x_1)\cdots\Phi(x_e)$, and note that $\Phi(x)$ is nondecreasing for $x > 1/2$, that $\Phi(1/2)=1$, $\Phi(3)=19/9 > 2$, and $\Lambda(1,\ldots,1)=2-2(2/3)^e+(1/2)^e < 2$.
This proves that a global minimum exists and lies in $(1/2,3)^e$: the closure of $(1/2,3)^e$ is compact and $\Lambda$ is continuous thereupon.
\end{step}
\begin{step}
Suppose that $(\sigmas)$ is a global minimizer of $\Lambda$.
Then the partial derivatives of $\Lambda$ must vanish there, whence
\begin{equation}\label{Rachel}
u(\sigma_k) \prod_{j=1}^e U(\sigma_j) + \frac{1}{2} v(\sigma_k) \prod_{j=1}^e  V(\sigma_j) = 1
\end{equation}
where $u(x)=\frac{x \Phi^\prime(x)}{\Phi(x)}$, $U(x)=\frac{3 \Phi(x)}{2 x}$, $v(x)=\frac{x \Psi^\prime(x)}{\Psi(x)}$, and $V(x)=\frac{3 \Psi(x)}{2 x}$ for $x>1/2$.
\end{step}
\begin{step}
We can assume $\sigmas \in (1,3)$: see \eqref{Rachel} and note that $u(x)=0$ for $x \in (1/2,1]$, $\frac{1}{2} v(x) V(x) < 1$ for $x \in (1/2,1]$, and $V(x) < 1$ for $x \in (1/2,3)$.
\end{step}
\begin{step}
It is not difficult to show that $U(x)$ strictly decreases and $V(x)$ strictly increases on $[1,3]$ with $U(3)=19/18$ and $V(1)=3/4$, and that $0 \leq u(x) < 1 \leq v(x)$ for $x \in [1,3]$.
Thus \eqref{Rachel} shows that we must have $u(\sigma_k) < \left(\frac{18}{19}\right)^e \left(1-\frac{3^e}{2^{2e+1}}\right)$ for all $k$.
This forces $u(\sigma_k) < 7/10$ for all $k$, and examination of the function $u$ shows that $u(x)\geq 7/10$ for $x\in[5/2,3]$, and so we must have $\sigma_k < 5/2$ for all $k$.

Now one can repeat the argument on the interval $[1,5/2]$ to show that every $\sigma_k < 2$, then repeat it again on $[1,2]$ to show $\sigma_k < 5/4$.  Further repetitions give $\sigma_k < 6/5$, $\sigma_k < 13/11$, and $\sigma_k < 7/6$.
Since $U(x) > 4/3$ while $v(x), V(x) \geq 0$ for $x \in (1,7/6)$, we have $u(\sigma_k) < (3/4)^e$ for all $k$, and since $u(x)\geq 8(x-1)/3$ for $x \in [1,7/6]$, this means that $\sigma_k < 1+\frac{3^{e+1}}{2^{2 e+3}}$ for all $k$.
\end{step}
\begin{step}
So we have $1 < \sigma_k < \min(\frac{7}{6},1+\frac{3^{e+1}}{2^{2 e +3}})$ for all $k$.
Consider the products in \eqref{Rachel}: since each $\sigma_j \in (1,7/6)$, we have $(4/3)^e < \prod_{j=1}^e U(\sigma_j) < (3/2)^e$ while $(3/4)^e < \prod_{j=1}^e V(\sigma_j) < (7/8)^e$.
We now claim that for a given $A \in [(4/3)^e,(3/2)^e]$ and $B \in [(3/4)^e,(7/8)^e],$ there is at most one solution $x \in (1,7/6)$ to
\[
A u(x) + \frac{1}{2}{B} v(x)=1,
\]
which will force $\sigma_1=\cdots=\sigma_e$.
For if we set $w(x)=A u(x)+(B/2) v(x)$, then we can show $w^\prime(x) > 0$ for $x \in (1,7/6)$: on this interval, we have $u(x)=4(x-1)/(3 x^2-4 x+2)$ and $v(x)=2/(2 x-1)$, and it is not difficult to show that $u^\prime(x) > 3/2$ and $v^\prime(x) > -4$, so that $w^\prime(x) > \frac{3 A}{2} -2 B \geq (3/2) (4/3)^e - 2 (7/8)^e > 0$.
\end{step}
\begin{step}\label{Ursula}
Now our problem is reduced to the single-variable minimization of $T(x)=\Lambda(x,\ldots,x)=-2\left(\frac{2 x}{3}\right)^e + 2 \Phi(x)^e + \Psi(x)^e$ for $x\in (1,1+\frac{3^{e+1}}{2^{2 e+3}})$.
It is not hard to see that $d T/d x$ vanishes if and only if  $x^{3 e} -\frac{3^e}{2^{e-2}} (x-1) (3 x^2-4 x+2)^{e-1} - \frac{3^e}{2^{2 e}} (2x-1)^{2 e-1}$ does.
Meanwhile we claim that the second derivative of $T$ is strictly positive on our interval: by dropping some nonnegative terms we see that 
\[
\frac{d^2T}{d x^2}(x)\geq -e(e-1)\frac{2^{e+1} x^{e-2}}{3^e} + e \frac{8(3-2 x)}{x^4} \Phi(x)^{e-1} + e \frac{3-4x}{x^4} \Psi(x)^{e-1}.
\]
Thus for $e=1$, the second derivative is at least $\frac{27-20 x}{x^4}$, which is strictly positive on our interval $\left(1,1+\frac{9}{32}\right)$.
For $e=2$, we can use the fact that $0 \leq \Psi(x) \leq 1 \leq \Phi(x)$ on our interval $\left(1,1+\frac{27}{128}\right)$ to show that the second derivative is at least $-\frac{16}{9}+\frac{54-40 x}{x^4}$, which is strictly positive on our interval.
Finally, if $e \geq 3$, we have $1+\frac{3^{e+1}}{2^{2 e+3}} < \frac{7}{6}$, and so on our interval $(1,1+\frac{3^{e+1}}{2^{2 e+3}})$ we have $\frac{8(3-2 x)}{x^4} > \frac{20}{7}$, $\Phi(x) \geq 1$, $\frac{3-4 x}{x^4} \geq -1$, and $0 \leq \Psi(x) \leq 1$, so that $\frac{d^2 T}{d x^2} (x) \geq -\frac{8}{9} \left(\frac{7}{9}\right)^{e-2} e(e-1) + \frac{13}{7} e > 0$.
This proves that there is a unique global minimum for this single-variable problem: the unique root $b_e$ of $x^{3 e} -\frac{3^e}{2^{e-2}} (x-1) (3 x^2-4 x+2)^{e-1} - \frac{3^e}{2^{2 e}}(2x-1)^{2 e-1}$ lying in $(1,1+\frac{3^{e+1}}{2^{2 e+3}})$.
Thus we have found that global minima are obtained precisely when $\sigma_1=\cdots=\sigma_e=b_e$ and $\tau_j \in \{\frac{1 - 2 b_e}{4} + \frac{m}{2} : m \in \Z\}$ for each $j \in \eseg$.
\end{step}

\section{Proof of Theorem \ref{Orlando}}\label{Albert}

For $A_e$ and $B_e$ as defined in Theorem \ref{Orlando}, we first show that $B_e < A_e$ for each $e\geq 1$.
Given the minimizing conditions described in Theorems \ref{Harold}\eqref{Karl} and \ref{Earl}\eqref{Ian}, it suffices to show that
\[
-\frac{2^{e+1}}{3^e} x^e + 2 \Omega\left(\frac{1}{x},0\right)^e + \Omega\left(\frac{1}{x},\frac{1}{2 x}\right)^e < -\frac{2^e}{3^e} x^e + 2 \Omega\left(\frac{1}{x},0\right)^e.
\]
for $x \in [1,3/2]$.  This follows if $\Omega(\frac{1}{x},\frac{1}{2 x}) < \frac{2 x}{3}$, or using the definition \eqref{Hilda} of $\Omega$, if $4 x^3 - 12 x^2 + 12 x- 3 > 0$ for $x \in [1,3/2]$, which is routine to show.

Now we show that $B_{e_1+e_2} < B_{e_1} B_{e_2}$ for any $e_1, e_2 \geq 1$.
We use a technical Lemma \ref{Henry} below, which provides bounds on the $B_e$.
It shows that if $e_1 \geq 5$ or $e_2 \geq 5$, then $B_{e_1} B_{e_2} > \sqrt[4]{2} > B_{e_1+e_2}$.
So we may confine ourselves to the case where $1 \leq e_1 \leq e_2 \leq 4$.
If we define $B_0=1$, then the bounds in Lemma \ref{Henry} also show us that
\[
\frac{B_1}{B_0} > \frac{B_2}{B_1} > \frac{B_3}{B_2} > \frac{B_4}{B_3} > \frac{B_5}{B_4}.
\]
Thus we note that $B_{e_1} B_{e_2} > B_{e_1-1} B_{e_2+1}$.
If $e_1+e_2 > 5$, we can repeat this argument to show that $B_{e_1} B_{e_2}$ is greater than $B_{e_1+e_2-5} B_5$, which we have already shown to exceed $B_{e_1+e_2}$.
On the other hand, if $e_1+e_2 \leq 5$, repetition of the same argument produces $B_{e_1} B_{e_2} > B_0 B_{e_1+e_2} = B_{e_1+e_2}$.

\begin{lemma}\label{Henry}
For each $e\geq 1$, let $B_e$ be the minimum asymptotic ratio of $\Lfour$ to $\Ltwo$ norm achievable by a family of $e$-variable quadratic multiplicative character polynomials as described in Theorem \ref{Earl}\eqref{Ian}.
Then
\begin{enumerate}[(i)]
\item\label{Barbara} $\sqrt[4]{103/89} < B_1 < \sqrt[4]{22/19}$,
\item $\sqrt[4]{86/65} < B_2 < \sqrt[4]{75/56}$,
\item $\sqrt[4]{142/95} < B_3 < \sqrt[4]{116/77}$,
\item $\sqrt[4]{100/61} < B_4 < \sqrt[4]{107/65}$,
\item\label{Helen} $\sqrt[4]{7/4} < B_5 < \sqrt[4]{128/73}$, and
\item $\sqrt[4]{7/4} < B_e < \sqrt[4]{2}$ for all $e \geq 6$.
\end{enumerate}
\end{lemma}
\begin{proof}
By Step \ref{Ursula} of the proof of Theorem \ref{Earl}\eqref{Ian} in Section \ref{George}, for each $e \geq 1$ the quantity $B_e^4$ is the minimum of function
\[
T(x)=-\frac{2^{e+1}}{3^e} x^e + 2 \Omega\left(\frac{1}{x},0\right)^e + \Omega\left(\frac{1}{x},\frac{1}{2 x}\right)^e
\]
on the interval $(1,1+\frac{3^{e+1}}{2^{2 e+3}})$, upon which the second derivative of $T$ is shown to be positive.
Thus if we can find $x_1$, $x_2$, $x_3$ in this interval with $x_1 < x_2 < x_3$ and $T(x_2) < T(x_1), T(x_3)$, then we will have shown that the minimizing value of $x$ lies in the interval $(x_1,x_3)$.
Then we can use $B_e^4 \leq T(x_2)$ for our upper bound and, by the monotonicity of $\Omega(\frac{1}{x},0)$ and $\Omega(\frac{1}{x},\frac{1}{2 x})$, we can use
\[
B_e^4 > -\frac{2^{e+1}}{3^e} x_3^e + 2 \Omega\left(\frac{1}{x_1},0\right)^e + \Omega\left(\frac{1}{x_1},\frac{1}{2 x_1}\right)^e
\]
as a lower bound.
We use this technique to prove bounds in \eqref{Barbara}--\eqref{Helen}. 
The calculations done by hand are tedious, so here we simply state choices of the triple $(x_1,x_2,x_3)$ that establish stricter bounds than the ones we claim above: for $B_1$, use $(55/52,128/121,73/69)$; for $B_2$, use $(18/17,17/16,16/15)$; for $B_3$, use $(21/20,20/19,19/18)$; for $B_4$, use $(26/25,25/24,24/23)$; and for $B_5$, use $(36/35,35/34,34/33)$.

Henceforth assume that $e \geq 6$, and let $b_e$ be the unique value in $(1,1+\frac{3^{e+1}}{2^{2 e+3}})$ such that $T(b_e)=B_e^4$.
Now note that $\Omega(\frac{1}{x},0) \geq 1$ and $\Omega(\frac{1}{x},\frac{1}{2 x}) \geq \frac{1}{2}$ for $x \geq 1$, and that $\frac{2 b_e}{3} < \frac{2^{15}+3^7}{3\cdot 2^{14}} < 1$, so that
\[
B_6^4 > -2\left(\frac{2^{15}+3^7}{3\cdot 2^{14}}\right)^6 + 2 + \frac{1}{2^6} > \frac{7}{4},
\]
and if $e \geq 7$, then 
\[
B_e^4 > -2\left(\frac{2^{15}+3^7}{3\cdot 2^{14}}\right)^7 + 2 > \frac{7}{4}.
\]
This proves our lower bound on $B_e$ when $e \geq 6$.

To prove our upper bound on $B_e$ when $e \geq 6$, write $b_e=1+c_e$ with $0 < c_e < \frac{3^{e+1}}{2^{2 e+3}}$, and use the definition \eqref{Hilda} of $\Omega$ and the fact that $1/(1+c_e) > 1-c_e$ to estimate $B_e=T(b_e)$ as 
\begin{align*}
B_e^4
& < -\frac{2^{e+1}}{3^e} + 2 (1+2 c_e^2)^e + \frac{1}{2^e} (1+c_e)^{2 e} \\
& \leq -\frac{2^{e+1}}{3^e} + 2 + 4 e c_e^2 + 2^{e+3} c_e^4 + \frac{1}{2^e} (1+c_e)^{2 e},
\end{align*}
where the second inequality follows from a crude approximation with the binomial expansion.  Now note that $2^{e+3} c_e^4 < 3^{4 e+4}/2^{7 e+9} < 2^e/(6 \cdot 3^e)$ and $4 e c_e^2 < e 3^{2 e+2}/2^{4 e+4} < (5 \cdot 2^e)/(4 \cdot 3^e)$, so that
\[
B_e^4 < 2 -\frac{7}{12} \cdot \frac{2^e}{3^e} + \frac{1}{2^e} (1+c_e)^{2 e},
\]
which will imply $B_e^4 < 2$ if we can show that $\left(\frac{3(1+c_e)^2}{4}\right)^e \leq \frac{7}{12}$.
Given that $c_e \leq 3^7/2^{15} < 1/10$, the quantity being raised to the $e$th power on the left hand side is less than $1$, so it suffices to show that $\left(\frac{3(11/10)^2}{4}\right)^6 \leq \frac{7}{12}$.
\end{proof}

\appendix
\section{Proofs of Auxiliary Results}\label{Sarah}
Here we collect, for the sake of completeness, technical results used in our proofs.
The first is a bound on a character sum used in the proofs of Theorems \ref{Gabriel}\eqref{Fabian}, \ref{Harold}\eqref{Yakov}, and \ref{Earl}\eqref{Nancy} in Section \ref{Tamar}.
\begin{lemma}\label{Xavier}
Let $\grp$ be a finite abelian group, $n$ a positive integer, and $\pi_1,\ldots,\pi_n\in\Hom(\Z,\grp)$ such that $\im\pi_1+\ldots+\im\pi_n$ is the internal direct sum of $\im\pi_1,\ldots,\im\pi_n$ in $\grp$.
Let $\pi \in \Hom(\Zn,\grp)$ with $\pi(u_1,\ldots,u_n)=\pi_1(u_1)+\ldots+\pi_n(u_n)$ for all $(u_1,\ldots,u_n) \in \Z^n$, and let $U=U_1\times\cdots\times U_n$ be a box in $\Zn$.  Then
\[
T=\sum_{\kappa,\lambda,\mu,\nu \in \grpchars} \Bigg|\sums{a,b,c,d \in U \\ a+b=c+d} \conj{\kappa(\pi(a)) \lambda(\pi(b))} \mu(\pi(c)) \nu(\pi(d))\Bigg| 
\]
is no greater than
\[
64^n \grpcard^4 \prod_{j=1}^n \max\left(1,\frac{\card{U_j}}{\card{\im\pi_j}}\right)^3 \prod_{j=1}^n (1+\log\card{\im\pi_j})^3
\]
\end{lemma}
\begin{proof}
Write $\subgrp=\bigoplus_{j=1}^n \im\pi_j$, so that $\subgrpchars=\prod_{j=1}^n \widehat{\im\pi_j}$.  Each character of $\subgrp$ extends to $\quotind$ characters of $\grp$, and for any $\eta \in \widehat{\im\pi_j}$, let $\eta^\prime \in \Zchars$ be $\eta\circ\pi_j$, so that
\[
T = \quotind^4 \prod_{j=1}^n \sum_{\kappa_j,\lambda_j,\mu_j,\nu_j \in \widehat{\im\pi_j}} \Bigg|\sums{a_j,b_j,c_j,d_j \in U_j \\ a_j+b_j=c_j+d_j}  \conj{\kappa_j^\prime(a_j)\lambda_j^\prime(b_j)}\mu_j^\prime(c_j)\nu_j^\prime(d_j)\Bigg|,
\]
and so it suffices to prove the bound when $n=1$ and $\pi$ is surjective.
In this case $\grp$ is a finite cyclic group, which we identify with $\ZmZ$ by identifying $\pi(1)$ with $1 \in \ZmZ$.
Then we set $\epsilon_a=\exp(2\pi i a/m)$ for every $a \in \ZmZ$, and note that $\grpchars$ is the set of maps $a \mapsto \epsilon_{x a}$ with $x \in \ZmZ$.
Thus
\[
T = \sum_{w,x,y,z \in \ZmZ} \Bigg|\sums{a,b,c,d \in U \\ a+b=c+d} \epsilon_{-w a -x b + y c + z d}\Bigg|.
\]
$U$ is a set of consecutive integers in $\Z$, and note that translation of $U$ does not influence the magnitude of the inner sum in $T$, so without loss of generality, we assume that $U=\{0,1,\ldots,\card{U}-1\}$.
Then reparameterize the outer sum in $T$ with $x^\prime=w-x$, $y^\prime=y-w$, $z^\prime=z-w$, and $w$ to obtain
\[
T = m \sum_{x^\prime,y^\prime,z^\prime \in \ZmZ} \Bigg|\sums{a,b,c,d \in U \\ a+b=c+d} \epsilon_{-x^\prime b + y^\prime c + z^\prime d}\Bigg|,
\]
which is not more than $64 m \max(m,\card{U})^3 (1+\log m)^3$ by \cite[Lemma 2.2]{Jedwab-Katz-Schmidt-a}.
\end{proof}
The next result is used in the proofs of Section \ref{Tamar} to understand the asymptotic behavior of the $\Ltwo$ norm for multiplicative character polynomials.
\begin{lemma}\label{Caroline}
Let $\pfam$ be a size-stable family of $e$-variable multiplicative character polynomials with $\Fqi$, $\Si$, $\ti$, and $\alphai$ respectively the field, support, translation, and arrangement of $\fio$ for each $\iota\in I$.
Then there is a $Q$ and an $N$ such that for all $\iota \in I$ with $\qi \geq Q$, we have $\card{\Si\cap(\ker\alphai-\ti)} \leq N$.
Thus $\card{\Si\cap\ker(\alphai-\ti)}/\card{\Si} \to 0$ and $\Ltt{\fio}/\card{\Si} \to 1$ as $\qi\to\infty$.
\end{lemma}
\begin{proof}
Suppose that the limiting sizes for our size-stable family $\pfam$ of multiplicative character polynomials are $\sigmas$.
For each $\iota \in I$, let $\chii$ be the character of $\fio$, so that $\fio(z)=\sum_{s \in \Si} \chii(\alphai(s+\ti)) z^s$, and let $\pio=\qi^{1/e}$, which is the characteristic of the field $\Fqi$ of $\fio$.
Since $\alphai$ in epimorphism, its restriction to each $\pio\times\cdots\times \pio$ cubical box in $\Ze$ is a bijection to $\Fqi$, and by the definition of size-stability, there is some $Q$ such that for every $\qi\geq Q$, the support $\Si=S_{\iota,1} \times \cdots \times S_{\iota,e}$ can be covered with $N=\prod_{j=1}^e \left(\floor{\sigma_j}+1\right)$ such cubes, each of which contains one point of $\ker\alphai-\ti$, so $\card{\Si\cap(\ker\alphai-\ti)} \leq N$.
Since the family is size-stable, $\card{\Si} \to \infty$ as $\qi \to \infty$.
The squared $\Ltwo$ norm of a polynomial is the sum of the squared magnitudes of its coefficients, and $\chii(\alphai(s+\ti))=0$ for $s \in \Si\cap(\ker\alphai-\ti)$ while $|\chii(\alphai(s+\ti))|=1$ for all other $s \in \Si$.
Thus $\Ltt{\fio}=\card{\Si\smallsetminus(\Si\cap(\ker\alphai-\ti))}$, and so $\Ltt{\fio}/\card{\Si} \to 1$ as $\qi\to\infty$.
\end{proof}
Recall from footnote \ref{Eric} of Section \ref{Francis} that we sometimes wish to obtain a Littlewood polynomial from a quadratic character polynomial $f(z)=\sum_{s \in S} \chi(\alpha(s+t)) z^s$, but $f$ may have some coefficients equal to $0$ because an extended nontrivial multiplicative character $\chi$ has $\chi(0)=0$.
More generally, if $\chi$ is a nontrivial multiplicative character (not necessarily quadratic), we may wish to obtain from $f$ a polynomial with coefficients of unit magnitude.
So we replace the zero coefficient for each $z^s$ such that $s \in S\cap(\ker\alpha-t)$ with a coefficient of unit magnitude.
We may choose each replacement coefficient independently of the others, and any polynomial $g$ resulting from such replacements is called a {\it unimodularization} of $f$.
The following corollary to Lemma \ref{Caroline} shows that unimodularizing all the polynomials in a size-stable family of multiplicative character polynomials does not affect asymptotic ratio $\normrat{f}$.
\begin{corollary}\label{Oswald}
Let $\pfam$ be a size-stable family of multiplicative character polynomials over fields $\fields$, and let $\gio$ be a unimodularization of $\fio$ for each $\iota\in I$.  If $r$ is a real number with $r\geq 2$ or if $r=\infty$, then $||\fio||_r/||\gio||_r \to 1$ as $\qi \to \infty$.
\end{corollary}
\begin{proof}
If $u \in \C$ with $|u|=1$ and if $s=(s_1,\ldots,s_e) \in \Z^e$, then $\Lr$ norm of $u z^s=u z_1^{s_1} \cdots z_e^{s_e}$ is $1$.
By Lemma \ref{Caroline} there is an $N$ and a $Q$ such that whenever $\qi\geq Q$, the two polynomials $\fio$ and $\gio$ differ by the sum of $N$ or fewer such monomials, and so by the $\Lr$ triangle inequality, the difference between $||\fio||_r$ and $||\gio||_r$ can not be greater in magnitude than $N$.
Now $||\gio||_r \geq ||\gio||_2=\sqrt{\card{\Si}}$ by monotonicity of $\Lr$ norms and the fact that the squared $\Ltwo$ norm of a polynomial is the sum of the squared magnitudes of its coefficients, and $\card{\Si}\to\infty$ as $\qi\to\infty$ for a size-stable family.
\end{proof}
The next result is used in Section \ref{Valerie} to find the limiting translations that globally minimize the asymptotic ratio of $\Lfour$ to $\Ltwo$ norm for quadratic character polynomials.
\begin{lemma}\label{Faye}
Let $x$ be a fixed nonzero value in $\R$ and let $y$ vary over $\R$.
\begin{enumerate}[(i)]
\item\label{Daphne} If $|x| \geq 2$, the function $\Omega(x,y)$, considered as a function of $y$, achieves a global minimum value of $0$ for $y \in \bigcup_{m \in \Z} [m |x|+1,(m+1)|x|-1]$ and for no other value of $y$.
\item\label{Edith} If $0 < |x| \leq 2$, the function $\Omega(x,y)$, considered as a function of $y$, achieves a global minimum value of
\[
\Omega\left(x,\frac{x}{2}\right)=\sum_{n \in \Z} \max\left(0,1-\left|\left(n+\frac{1}{2}\right)x\right|\right)^2
\]
for $y\in \left\{x\left(m+\frac{1}{2}\right) : m \in \Z\right\}$ and for no other value of $y$.
\end{enumerate}
\end{lemma}
\begin{proof}
For part \eqref{Daphne}, note that all the terms of $\Omega(x,y)$ are nonnegative.  
Since $\Omega(x,y)=\Omega(-x,y)$, we may assume without loss of generality that $x \geq 2$, and then the term $\max(0,1-|x n-y|)^2$ is nonvanishing if and only if $\frac{y-1}{x} < n < \frac{y+1}{x}$.
Thus we can obtain a global minimum value of $0$ if we can arrange that no integer lie in the interval $(\frac{y-1}{x},\frac{y+1}{x})$.
If $m$ is the greatest integer lying below this interval (so that $y \geq m x+1$), then for the next integer $m+1$ to lie above the interval, it is necessary and sufficient that $y\leq (m+1)x-1$.

For part \eqref{Edith}, it is clear from the definition \eqref{Hilda} of $\Omega$ that $\Omega(-x,y)=\Omega(x,y)$, $\Omega(x,-y)=\Omega(x,y)$, and $\Omega(x,y)=\Omega(x,y+x)$.
So without loss of generality we may restrict our attention to the case where $0 < x \leq 2$ and $0 \leq y \leq x/2$.
In this case
\[
\Omega(x,y) = \sum_{\ceil{(y-1)/x} \leq n \leq 0} (y-1-n x)^2 + \sum_{0 < n \leq \floor{(y+1)/x}} (y+1-n x)^2,
\]
and we reparameterize the sums to obtain
\[
\Omega(x,y) = \sum_{0 \leq n \leq \floor{(1-y)/x}} (y-1+n x)^2 + \sum_{0 \leq n \leq \floor{(1+y-x)/x}} (y-x+1-n x)^2,
\]
and calculate
\begin{align*}
\frac{\partial}{\partial y} \Omega(x,y) 
& = \! \sum_{0 \leq n \leq \floor{(1-y)/x}} 2(y-1+n x) + \! \sum_{0 \leq n \leq \floor{(1+y-x)/x}} 2(y-x+1-n x) \\
& = 2 \Bigfloor{\frac{y+1}{x}} (2 y-x)+ \sum_{\floor{(1+y-x)/x}  < n \leq \floor{(1-y)/x}} 2(y-1+n x),
\end{align*}
because $(1-y)/x$ is greater than or equal to $(1+y-x)/x$, and note for the remainder of this proof that their difference is at most $1$.
Since $0 \leq y \leq x/2 \leq 1$, we can see that both terms in the last expression for our partial derivative are nonpositive, with the summation over $n$ strictly negative if $y < x-1$, and the other term is strictly negative if $x-1 \leq y < x/2$.  Thus our partial derivative is strictly negative for $0 \leq y < x/2$, and so for our ranges of $x$ and $y$ values, the unique minimum is obtained when $y=x/2$.
\end{proof}
The last two results are used in Section \ref{Valerie} to show that a large limiting size will make the ratio of $\Lfour$ to $\Ltwo$ norm large.
\begin{lemma}\label{Katherine}
If $\pfam$ is a size-stable family, with limiting size $\sigma$, of additive character polynomials over fields $\fields$, then 
\[
\liminf_{\qi \to \infty} \qnormrat{\fio} \geq \Omega\left(\frac{1}{\sigma},0\right).
\]
\end{lemma}
\begin{proof}
For each $\iota\in I$, let $\fio$ have character $\psii$, support $\Si$, translation $t_i$, and arrangement $\alphai$, so that $\fio(z)=\sum_{s \in \Si} \psii(\alphai(s+t_i)) z^s$
When we confine the values of $z$ to the complex unit circle, we have
\[
\conj{\fio(z)} = \sum_{s \in \Si} \conj{\psii(\alphai(s+t_i))} z^{-s}.
\]
We can consider $\fio(z)$ and $\conj{\fio(z)}$ formally as elements of $\C[z,z^{-1}]$, and view $\Lff{\fio}$ as the sum of the squared magnitudes of the coefficients of $\fio(z) \conj{\fio(z)}$.
The coefficient of $z^s$ in $\fio(z) \conj{\fio(z)}$ is
\[
\sums{u,v \in \Si \\ u-v= s} \psii(\alphai(u+t_i)) \conj{\psii(\alphai(v+t_i))}.
\]
Since $\alphai$ is an epimorphism from $\Z$ to $\F_{\qi}^*$, we see that $\psii(\alphai(u+t_i))=\psii(\alphai(v+t_i))$ whenever $u\equiv v\pmod{\qi-1}$.
Thus if $s \equiv 0 \pmod{\qi-1}$, the coefficient of $z^s$ in $\fio(z) \conj{\fio(z)}$ is equal to $\card{\Si\cap(s+\Si)}$.
Since $\Lff{\fio}$ is the sum of the squared magnitudes of the coefficients of $\fio(z) \conj{\fio(z)}$ while $\Ltt{\fio}$ is the coefficient of $z^0$ of the same, we have
\[
\qnormrat{\fio} \geq \frac{\sum_{n \in \Z} \card{\Si\cap(n(\qi-1)+\Si)}^2}{\card{\Si}^2},
\]
and then we note that $\card{\Si\cap(n(\qi-1)+\Si)}=\max(0,\card{\Si}-|n(\qi-1)|)$ and apply the size-stability limit $\card{\Si}/(\qi-1) \to \sigma$ as $\qi\to\infty$.
\end{proof}
\begin{lemma}\label{Laura}
If $\pfam$ is a size-stable family, with limiting sizes $\sigmas$, of $e$-variable multiplicative character polynomials over fields $\fields$, then
\[
\liminf_{\qi \to \infty} \qnormrat{\fio} \geq \prod_{j=1}^e \Omega\left(\frac{1}{\sigma_j},0\right).
\]
\end{lemma}
\begin{proof}
For each $\iota\in I$, let $\fio$ have character $\chii$, support $\Si \subseteq \Ze$, translation $t_i \in \Z^e$, and arrangement $\alphai$, so that $\fio(z_1,\ldots,z_e)=\sum_{s \in \Si} \chii(\alphai(s+t_i)) z^s$, where we write $z^s$ for $z_1^{s_1} \cdots z_e^{s_e}$ when $s=(s_1,\ldots,s_e)$.
Our proof runs the same as that of the previous lemma for additive character polynomials once we replace $\psii$ with $\chii$, but we must take care of the fact that $\chii(\alphai(s+t_i))=0$ when $s\in -t_i+\ker\alphai$; otherwise the coefficients are of unit magnitude.
And of course the polynomials are in $e$ variables and the coefficients have periodicity $p$ in each direction.
Thus if we define $\Vi=\Si\smallsetminus(-t_i+\ker\alphai)$ we have
\[
\qnormrat{\fio} \geq \frac{\sum_{n \in \Ze} \card{\Vi\cap(n \pio+\Vi)}^2}{\card{\Vi}^2},
\]
but Lemma \ref{Caroline} can be used to show that the ratio $\card{\Vi\cap(n \pio+\Vi)}/\card{\Vi}$ has the same limit as $\card{\Si\cap(n \pio+\Si)}/\card{\Si}$ as $\qi\to\infty$.
\end{proof}
\section*{Acknowledgements}
The author thanks Jonathan Jedwab and Kai-Uwe Schmidt for helpful suggestions on the presentation of these results.

\end{document}